\documentclass[12pt]{amsart}
\usepackage{a4wide}
\usepackage[T1]{fontenc}
\usepackage{amssymb,amsmath,amsthm,latexsym}
\usepackage{mathrsfs}
\usepackage[usenames,dvipsnames]{color}
\usepackage{euscript}
\usepackage{graphicx}
\usepackage{mdwlist}
\usepackage{enumerate}
\usepackage{mathtools,dsfont,wasysym}
\usepackage{stmaryrd}
\usepackage[all]{xy}
\usepackage{multicol}
\usepackage{hyperref}
\hypersetup{colorlinks=true, linkcolor=blue, citecolor=blue, urlcolor=blue}

\newtheorem{theorem}{Theorem}[section]
\newtheorem{lemma}[theorem]{Lemma}

\newtheorem{proposition}[theorem]{Proposition}
\newtheorem{fact}[theorem]{Fact}
\newtheorem{claim}[theorem]{Claim}
\newtheorem{problem}[theorem]{Problem}
\newcounter{maintheorem}

\newtheorem{mainth}[maintheorem]{Theorem}
\theoremstyle{remark}
\newtheorem{remark}[theorem]{Remark}
\theoremstyle{definition}
\newtheorem{definition}[theorem]{Definition}

\numberwithin{equation}{section}
\makeatother

\newcommand{\R}{\mathbb{R}}

\newcommand{\e}{\varepsilon}
\newcommand{\p}{\varphi}

\newcommand{\nn}[1]{{\left\vert\kern-0.25ex\left\vert\kern-0.25ex\left\vert #1 
\right\vert\kern-0.25ex\right\vert\kern-0.25ex\right\vert}}
\newcommand{\cut}{\mathord{\upharpoonright}}
\renewcommand{\leq}{\leqslant}
\renewcommand{\geq}{\geqslant}

\newcommand{\ro}{\varrho}
\renewcommand{\P}{\mathcal{P}}
\newcommand{\C}{\mathcal{C}}
\newcommand{\M}{\mathcal{M}}
\newcommand{\D}{\mathcal{D}}

\newcommand{\A}{\mathcal{A}}
\newcommand{\Fr}{{\mathcal{F}_\ro}}
\newcommand{\FA}{\mathcal{F_A}}
\newcommand{\KA}{\mathcal{K_A}}
\newcommand{\K}{\mathcal{K}}
\newcommand{\Kr}{{\mathcal{K}_\ro}}
\newcommand{\X}{\mathcal{X}}
\newcommand{\Xr}{{\mathcal{X}_\ro}}
\newcommand{\Xl}{{\mathcal{X}_\Lambda}}
\renewcommand\qedsymbol{$\blacksquare$} 

\newcounter{smallromans}

\newenvironment{romanenumerate}
{\begin{list}{{\normalfont\textrm{(\roman{smallromans})}}}
  {\usecounter{smallromans}\setlength{\itemindent}{0cm}
   \setlength{\leftmargin}{5.5ex}\setlength{\labelwidth}{5.5ex}
   \setlength{\topsep}{.5ex}\setlength{\partopsep}{.5ex}
   \setlength{\itemsep}{0.1ex}}}
{\end{list}}

\begin{document}
\title[An Asplund space with norming M-basis that is not WCG]{An Asplund space with norming Marku\v{s}evi\v{c} basis\\ that is not weakly compactly generated}

\author[P.~H\'ajek]{Petr H\'ajek}
\address[P.~H\'ajek]{Department of Mathematics\\Faculty of Electrical Engineering\\Czech Technical University in Prague\\Technick\'a 2, 166 27 Prague 6\\ Czech Republic}
\email{hajek@math.cas.cz}

\author[T.~Russo]{Tommaso Russo}
\address[T.~Russo]{Institute of Mathematics\\ Czech Academy of Sciences\\ \v{Z}itn\'a 25, 115 67 Prague 1\\ Czech Republic; and Department of Mathematics\\Faculty of Electrical Engineering\\Czech Technical University in Prague\\Technick\'a 2, 166 27 Prague 6\\ Czech Republic}
\email{russo@math.cas.cz, russotom@fel.cvut.cz}

\author[J.~Somaglia]{Jacopo Somaglia}
\address[J.~Somaglia]{Dipartimento di Matematica ``F. Enriques'' \\Universit\`a degli Studi di Milano\\Via Cesare Saldini 50, 20133 Milano\\Italy}
\email{jacopo.somaglia@unimi.it}

\author[S.~Todor\v{c}evi\'{c}]{Stevo Todor\v{c}evi\'{c}}
\address[S.~Todor\v{c}evi\'{c}]{Department of Mathematics\\ University of Toronto\\ Toronto\\ Ontario M5S 2E4\\ Canada; Institut de Math\'ematiques de Jussieu\\ Paris\\ France; and Matemati\v{c}ki Institut\\ SANU\\ Belgrade\\ Serbia}
\email{stevo@math.utoronto.ca, stevo.todorcevic@imj-prg.fr, stevo.todorcevic@sanu.ac.rs}

\thanks{P.~H\'ajek was supported in part by OPVVV CAAS CZ.02.1.01/0.0/0.0/16$\_$019/0000778.\\
Research of T.~Russo was supported by the GA\v{C}R project 20-22230L; RVO: 67985840 and by Gruppo Nazionale per l'Analisi Matematica, la Probabilit\`a e le loro Applicazioni (GNAMPA) of Istituto Nazionale di Alta Matematica (INdAM), Italy.\\ 
J.~Somaglia was supported  by Universit\`a degli Studi di Milano, Research Support Plan 2019 and by Gruppo Nazionale per l'Analisi Matematica, la Probabilit\`a e le loro Applicazioni (GNAMPA) of Istituto Nazionale di Alta Matematica (INdAM), Italy.\\
Research of S.~Todor\v{c}evi\'{c} is partially supported by grants from NSERC(455916) and CNRS(UMR7586).}

\dedicatory{Dedicated to Clemente Zanco on the occasion of his retirement}
\date{\today}

\begin{abstract} We construct an Asplund Banach space $\mathcal{X}$ with a norming Marku\v{s}evi\v{c} basis such that $\mathcal{X}$ is not weakly compactly generated. This solves a long-standing open problem from the early nineties, originally due to Gilles Godefroy. \emph{En route} to the proof, we construct a peculiar example of scattered compact space, that also solves a question due to Wies\l aw Kubi\'s and Arkady Leiderman.
\end{abstract}
\maketitle

\section{Introduction}

The crystallisation, in the mid-sixties, of the notions of projectional resolution of the identity (PRI, for short) \cite{Lind refl} and of weakly compactly generated Banach space (WCG) \cite{AmirLind} opened the way to a spectacular development in Banach space theory, leading to a structural theory for many classes of non-separable Banach spaces. Just to mention some advances, we refer, \emph{e.g.}, to \cite{ArMe WLD}, \cite{DFJP}, \cite{Finet}, \cite{JL}, \cite{Rosenthal}, \cite{Talagrand}, \cite{Zizler}. Such a theory is tightly connected to differentiability \cite{AK}, \cite{FGZ1}, \cite{FGZ2}, \cite{FHZ}, \cite{FMZ}, \cite{JZ2}, classes of compacta \cite{AMN}, \cite{BRW}, \cite{BS}, \cite{Corson}, \cite{CorsonLind}, \cite{Farmaki}, \cite{Kalenda survey}, \cite{Lind wc}, combinatorics \cite{Argyros}, \cite{DLT1}, \cite{DLT2}, \cite{LT2}, \cite{Ost}, \cite{Shelah}, \cite{T scheme}, \cite{T MA}.

An important tool in the area was introduced by Fabian \cite{Fabian dual LUR}, who used Jayne--Rogers selectors \cite{JR} to show that every weakly countably determined Asplund Banach space is indeed WCG. Jayne--Rogers selectors were also deeply involved, together with Simons' lemma \cite{Simons}, \cite{Godefroy}, in the proof that the dual of every Asplund space admits a PRI, \cite{FG}. The techniques of \cite{Fabian dual LUR} also used ingredients from \cite{JZ2}, where it is shown, among others, that WCG Banach spaces with a Fr\'echet smooth norm admit a shrinking M-basis. Results of this nature led to the conjecture that Asplund Banach spaces with a norming M-basis are WCG. This question is originally due to Godefroy, who, at the times when \cite{DGZ} was in preparation, conjectured that a similar use of Jayne--Rogers selectors might produce a linearly dense weakly compact subset, in presence of a norming M-basis.

\begin{problem}[G.~Godefroy]\label{Problem} Let $\X$ be an Asplund space with a norming Marku\v{s}evi\v{c} basis. Must $\X$ be weakly compactly generated?
\end{problem}
The problem was subsequently recorded in various articles and books, see, \emph{e.g.}, \cite{AP}, \cite[p.~211]{HMVZ}, \cite[Problem~112]{GMZ}. The main result of our paper is a negative answer to this problem, in the form of the following result.

\begin{mainth}\label{MTh: Xr} There exists an Asplund space $\X$ with a $1$-norming M-basis such that $\X$ is not WCG.
\end{mainth}
The proof of Theorem \ref{MTh: Xr} will be given in Section \ref{Sec: Proof Th A}. As a matter of fact, the M-basis that we construct there is additionally an Auerbach basis, see Remark \ref{Rmk: M-basis is Auerbach}. Moreover, a small elaboration over the argument also produces a counterexample with a long monotone Schauder basis (Theorem \ref{Th: Xr with basis}).\smallskip

Since the result \cite{AmirLind} that WCG Banach spaces admit an M-basis, and, therefore, reflexive spaces have a shrinking basis, it readily became clear that M-bases with additional properties would have been instrumental in the characterisation of several classes of Banach spaces, \cite[Chapter 6]{HMVZ}, \cite{Zizler}. In particular, it was natural to ask which class of Banach spaces is characterised by the presence of a norming M-basis. This led to the famous question, due to John and Zizler, whether every WCG Banach space admits a norming M-basis \cite{JZ3}, that was recently solved in the negative by the first-named author, \cite{Hajek}. In this sense, Problem \ref{Problem} can be considered as a converse to the said John's and Zizler's question.

As it turns out, there is an elegant characterisation of Banach spaces that admit a shrinking M-basis, in the form of the following result, due to the efforts of many mathematicians, \cite{Fabian dual LUR}, \cite{Valdivia}, \cite{OV}, \cite{JZ1}, \cite{JZ2}.  We refer to \cite[Theorem 6.3]{HMVZ}, or \cite[Theorem 8.3.3]{Fabian book} for a proof.

\begin{theorem} \label{Th: shrinking M-basis} For a Banach space $\X$, the following are equivalent:
\begin{romanenumerate}
    \item $\X$ admits a shrinking M-basis;
    \item $\X$ is WCG and Asplund;
    \item $\X$ is WLD and Asplund;
    \item $\X$ is WLD and $\X^*$ has a dual LUR norm;
    \item $\X$ is WLD and it admits a Fr\'echet smooth norm.
\end{romanenumerate}
\end{theorem}

Problem \ref{Problem} is clearly closely related to this result, since it amounts to asking whether the assumption in (ii) that $\X$ is WCG could be replaced by the existence of a norming M-basis. \smallskip

Our construction in Theorem \ref{MTh: Xr} heavily depends on the existence of a peculiar scattered compact space, whose properties we shall record in Theorem \ref{MTh: Kr} below. Before its statement, we need one piece of notation.

Given a set $S$, we identify the power set $\P(S)$ with the product $\{0,1\}^S$, via the canonical correspondence $A \leftrightarrow 1_A$ ($A\subseteq S$). Since $\{0,1\}^S$ is a compact topological space in its natural product topology, this identification allows us to introduce a compact topology on $\P(S)$. Throughout our article, any topological consideration relative to $\P(S)$ will refer to the said topology, that we shall refer to as the \emph{product}, or \emph{pointwise}, topology. 

\begin{mainth}\label{MTh: Kr} There exists a family $\Fr\subseteq [\omega_1]^{<\omega}$ of finite subsets of $\omega_1$ such that $\Kr:=\overline{\Fr}$ has the following properties:
\begin{romanenumerate}
\item $\{\alpha\}\in\Kr$ for every $\alpha<\omega_1$,
\item $[0,\alpha)\in\Kr$ for every $\alpha\leq\omega_1$,
\item if $A\in\Kr$ is an infinite set, then $A=[0,\alpha)$ for some $\alpha\leq\omega_1$,
\item $\Kr$ is scattered.
\end{romanenumerate}
\end{mainth}

The subscript $\ro$ in our notation for the family $\Fr$ reflects the r\^{o}le of the choice of a $\ro$-function, a rather canonical semi-distance on $\omega_1$, in the construction of the family $\Fr$. $\ro$-functions were introduced in \cite{T Acta} for a study of the way Ramsey's theorem fails in the uncountable context. They appeared already in Banach space constructions, see, \emph{e.g.}, \cite{ALT}, \cite{AvTo}, \cite{LT}, \cite{LT2}. We refer to \cite{T}, \cite{AT}, \cite{Bekkali} for a detailed presentation of this theory and further applications in several areas. $\ro$-functions are also tightly related to \emph{construction schemes}, \cite{T scheme}, that also proved very useful in non-separable Banach space theory, see, \emph{e.g.}, \cite{Lopez}.\smallskip

In conclusion to this section, we briefly describe the organisation of the paper. Section \ref{Sec: preliminaries} contains a revision of the notions from non-separable Banach space theory that are relevant to our paper. The proof of Theorem \ref{MTh: Kr}, together with a quick revision of the necessary results concerning $\ro$-functions, will be given in Section \ref{Sec: Proof of Th B}. Section \ref{Sec: Consequences of Th B} is independent from the argument in Section \ref{Sec: Proof of Th B}, as it only depends on the \emph{statement} of Theorem \ref{MTh: Kr}. Apart for the proof of Theorem \ref{MTh: Xr}, we observe there that the compact space in Theorem \ref{MTh: Kr} also offers an interesting example for the theory of semi-Eberlein compacta, solving a problem from \cite{KL}. Finally, in Section \ref{S: adequate}, we discuss the main problem in the $\C(\K)$ case, where $\K$ is an adequate compact; in particular, we show that $\C(\K)$ has a $1$-norming M-basis, whenever $\K$ is adequate. Finally, a typographical note: the symbol \qedsymbol\ denotes the end of a proof, while, in nested proofs, we use $\square$ for the end of the inner proof.

\section{Preliminary definitions and results}\label{Sec: preliminaries}
\subsection{General conventions}
Our notation concerning Banach spaces is standard, as in most textbooks in Banach space theory; we refer, \emph{e.g.}, to \cite{akbook}, \cite{FHHMZ}. All our results in the paper are valid for Banach spaces over either the real or the complex fields, with the same proofs. By a \emph{subspace} of a Banach space we understand a closed, linear subspace.\smallskip

We indicate by $|S|$ the cardinality of a set $S$. For a set $S$ and a cardinal number $\kappa$, we write $[S]^\kappa=\{A\subseteq S\colon |A|=\kappa\}$ and $[S]^{<\kappa}=\{A\subseteq S\colon |A|<\kappa\}$. We denote by $\omega$ the first infinite ordinal and by $\omega_1$ the first uncountable one. We also adopt the convention to regard cardinal numbers as initial ordinals; hence, we write $\omega$ for $\aleph_0$, $\omega_1$ for $\aleph_1$, and so on. If $A$ and $B$ are subsets of an ordinal $\alpha$, we write $A<B$ meaning that $a<b$ whenever $a\in A$ and $b\in B$. Given a pair of ordinals $\alpha\leq\beta$, we denote by $[\alpha,\beta]$ and $[\alpha,\beta)$ the sets comprising all ordinals $\gamma$ such that $\alpha\leq\gamma \leq\beta$ and $\alpha\leq\gamma <\beta$, respectively. We equip the intervals $[\alpha,\beta]$ and $[\alpha,\beta)$ with the canonical order topology which turns every interval $[\alpha,\beta]$ into a compact space and makes $[\alpha,\beta)$ compact if and only if $\beta$ is a successor ordinal. We follow this notation also for the intervals $[0,\alpha]$ and $[0,\alpha)$. According to the standard definition of ordinals, the interval $[0,\alpha)$ coincides with the ordinal $\alpha$; however, we shall mostly keep the notation $[0,\alpha)$ to stress the topological structure rather than, say, the r\^{o}le of index set. \smallskip

Next, we shall record some basic notions concerning biorthogonal systems in Banach spaces. A \emph{biorthogonal system} in a Banach space $\X$ is a system $\{x_\gamma; \p_\gamma\}_{\gamma\in\Gamma}$, with $x_\gamma\in\X$ and $\p_\gamma\in\X^*$, such that $\langle\p_\alpha, x_\beta \rangle=\delta_{\alpha,\beta}$ ($\alpha,\beta\in \Gamma$). A biorthogonal system is \emph{fundamental} (or \emph{complete}) if ${\rm span}\{x_\gamma\}_{\gamma\in\Gamma}$ is dense in $\X$; it is \emph{total} when ${\rm span}\{\p_\gamma\} _{\gamma\in\Gamma}$ is $w^*$-dense in $\X^*$. A \emph{Marku\v{s}evi\v{c} basis} (henceforth, M-basis) is a fundamental and total biorthogonal system. An \emph{Auerbach basis} is an M-basis such that $\|x_\gamma\|=\|\p_\gamma\|=1$, for every $\gamma\in\Gamma$.

M-bases exist in many Banach spaces, in particular in every separable one, \cite{Markushevich}, and many classes of Banach spaces admit characterisations in terms of M-bases (see, \emph{e.g.}, \cite[Chapter 6]{HMVZ}). The simplest example of a Banach space without M-bases is $\ell_\infty$, \cite{Joh70}; on the other hand, $\ell_\infty$ admits a fundamental biorthogonal system \cite{DJ} (and, of course, also a total one).

An M-basis $\{x_\gamma; \p_\gamma\}_{\gamma\in\Gamma}$ is \emph{shrinking} if ${\rm span}\{\p_\gamma\} _{\gamma\in\Gamma}$ is dense in $\X^*$. As we saw in Theorem \ref{Th: shrinking M-basis}, this is a rather strong notion, as it implies that $\X$ is both Asplund and WCG. The M-basis $\{x_\gamma; \p_\gamma\}_{\gamma\in\Gamma}$ is  \emph{$\lambda$-norming} ($0<\lambda \leq1$) if
$$\lambda\|x\|\leq \sup\big\{|\langle\p,x\rangle|\colon \p\in {\rm span}\{\p_\gamma\}_{\gamma\in\Gamma},\, \|\p\|\leq 1 \big\} \qquad(x\in\X),$$
namely, if $\overline{\rm span}\{\p_\gamma\} _{\gamma\in\Gamma}$ is a $\lambda$-norming subspace for $\X$. $\{x_\gamma; \p_\gamma\}_{\gamma\in\Gamma}$ is \emph{norming} if it is $\lambda$-norming, for some $\lambda>0$.

An important example of a Banach space that admits no norming M-basis is $\C([0,\omega_1])$, as proved by Alexandrov and Plichko \cite{AP}. On the other hand, $\C([0,\omega_1])$ admits a strong and countably $1$-norming M-basis, \cite{AP} (see also \cite[Theorem 5.25]{HMVZ}). Recall that an M-basis is \emph{countably $\lambda$-norming} if $\left\{\p\in\X^* \colon \{\gamma\in\Gamma\colon \langle\p,x_\gamma\rangle\neq0\} \text{ is countable} \right\}$ is a $\lambda$-norming subspace. Kalenda \cite{Kalenda renorm Valdivia} proved that $\C_0([0,\omega_1))$ (which is isomorphic to $\C([0,\omega_1])$) admits no countably $1$-norming M-basis. Moreover, $\C([0,\omega_2])$ admits no countably norming M-basis, \cite{Kalenda omega2}.

\subsection{Some classes of compact spaces}
A topological space $\K$ is \emph{scattered} if every its closed subspace has an isolated point; in other words, $\K$ contains no non-empty perfect subset. Every scattered compact is \emph{zero-dimensional}, \emph{i.e.}, it admits a basis consisting of clopen sets \cite[Theorem 29.7]{Willard}. A compact space is countable if and only if it is metrisable and scattered \cite[Lemma VI.8.2]{DGZ}; moreover, scattered compacta are closed under continuous images, \cite{Rudin}.\smallskip

Given an arbitrary set $\Gamma$, the \emph{$\Sigma$-product} of real lines is
\begin{equation*}
\Sigma(\Gamma)=\{x\in [0,1]^{\Gamma}\colon|\{\gamma\in \Gamma\colon x(\gamma)\neq 0\}|\leq \omega\}.
\end{equation*}
(Here, and throughout the paper, we consider the product topology on the space $[0,1]^\Gamma$.) For an element $x\in\Sigma(\Gamma)$, we call the set ${\rm supp}(x):=\{\gamma\in \Gamma\colon x(\gamma)\neq 0\}$ the \emph{support} of $x$. A topological space $\K$ is \emph{Fr\'echet--Urysohn} if for every subset $A$ of $\K$ and every $p\in\overline{A}$, there exists a sequence in $A$ that converges to $p$. It is easy to see that every $\Sigma$-product is Fr\'echet--Urysohn, \cite[Lemma 1.6]{Kalenda survey}. It is also clear that $\Sigma(\Gamma)$ is dense in $[0,1]^\Gamma$ and it is \emph{countably closed} (in the sense that the closure, in $[0,1]^\Gamma$, of every countable subset of $\Sigma(\Gamma)$ is contained in $\Sigma(\Gamma)$).

A compact space is \emph{Eberlein} if it is homeomorphic to a weakly compact subset of some Banach space. In their seminal paper \cite{AmirLind}, Amir and Lindenstrauss proved that every Eberlein compact is homeomorphic to a weakly compact subset of $c_0(\Gamma)$, for some set $\Gamma$. Let us also recall, in passing, that every Eberlein compact is homeomorphic to a weakly compact subset of a reflexive Banach space, \cite{DFJP}.

A compact space $\K$ is \emph{Corson} if $\K$ is homeomorphic to a subset of $\Sigma(\Gamma)$, for some set $\Gamma$. Therefore, Corson compacta are Fr\'echet--Urysohn; in particular, $[0,\omega_1]$ is not Corson. Of course, every Eberlein compact is Corson. The converse implication is false, \cite[Example 5.1]{BRW}, \cite{AlsterPol}, \cite{Talagrand}; however, scattered Corson compacta are Eberlein, \cite{Alster}.

A compact space is \emph{Valdivia} if it is homeomorphic to a compact subspace $\K$ of $[0,1]^{\Gamma}$ such that $\K\cap \Sigma(\Gamma)$ is dense in $\K$. If $\K$ is Valdivia and $h\colon \K \to [0,1]^{\Gamma}$ is an embedding that witnesses this, namely if $h(\K)\cap\Sigma(\Gamma)$ is dense in $h(\K)$, we set $\Sigma(\K):=h^{-1}(\Sigma(\Gamma))$. $\Sigma(\K)$ is then called a \emph{$\Sigma$-subset} of $\K$ and, by definition, it is dense in $\K$. For further information on Valdivia compacta we refer to the very detailed survey \cite{Kalenda survey} and the references therein. \smallskip

We denote $\P(\Gamma)$ the power set of a set $\Gamma$, which, throughout our article, we identify with the product $\{0,1\}^\Gamma$, in the canonical way. This permits us to introduce a compact topology on $\P(\Gamma)$, that we also call the product, or pointwise, topology. With the said topology, it is easy to see that the correspondence $\alpha\mapsto [0,\alpha)$ defines an embedding of $[0,\omega_1]$ into $\P(\omega_1)$; in particular, $[0,\omega_1]$ is a Valdivia compactum.

As it turns out, $[0,\omega_1]$ is the archetypal example of a Valdivia compact space that is not Corson. Indeed, Deville and Godefroy \cite{DG} proved the very elegant result that a Valdivia compact space is Corson if and only if it contains no copy of $[0,\omega_1]$. Kalenda \cite{Kalenda embed} (see also \cite[Proposition 3.12]{Kalenda survey}) generalised this result as follows: if $\K\subseteq[0,1]^{\omega_1}$ is such that $\K\cap \Sigma(\omega_1)$ is dense in $\K$ and $p\in\K\setminus \Sigma(\omega_1)$, there exists an embedding $\p\colon [0,\omega_1]\to \K$ with
\begin{romanenumerate}
    \item $\p(\alpha)\in \K \cap\Sigma(\omega_1)$, for $\alpha<\omega_1$,
    \item ${\rm supp}(\p(\alpha))\subseteq {\rm supp}(\p(\beta))$, for $\alpha<\beta \leq\omega_1$,
    \item $\p(\omega_1)=p$.
\end{romanenumerate}\smallskip

One more crucial property of Valdivia compacta is that they admit a `good' system of retractions, \emph{e.g.}, \cite{KubisMicha}. The canonical way to build them is to use the following lemma, based on a rather standard closing-off argument; see, \emph{e.g.}, \cite[Lemma 1.2]{AMN}, \cite[Lemma VI.7.5]{DGZ}, or \cite[Lemma 19.10]{KKL}. For an element $x\in [0,1]^\Gamma$ and $J\subseteq\Gamma$, we define $x\cut_J$ by
$$x\cut_J(\gamma)= \begin{cases} x(\gamma) & \gamma\in J \\ 0 & \gamma\notin J.\end{cases}$$
\begin{lemma}\label{Lemma: closing-off} Let $\K\subseteq[0,1]^\Gamma$ be a compact space such that $\K\cap \Sigma(\Gamma)$ is dense in $\K$. For any infinite set $I\subseteq\Gamma$ there exists a set $J\subseteq \Gamma$ with $I\subseteq J$, $|I|=|J|$ and such that
$$x\in\K \implies x\cut_J \in\K.$$
\end{lemma}

\subsection{Non-separable Banach spaces}\label{Sec: Classes of B spaces}
A Banach space $\X$ is \emph{Asplund} if every convex continuous function defined on a convex open subset $U$ of $\X$ is Fr\'echet differentiable on a dense $G_\delta$ subset of $U$. It is well known that $\X$ is Asplund if and only if every separable subspace of $\X$ has a separable dual, see \cite[Theorem I.5.7]{DGZ}. In particular, we can see that the class of Asplund spaces is closed under taking subspaces and quotients. In the subclass of $\C(\K)$ spaces, a $\C(\K)$ space is Asplund if and only if the compact $\K$ is scattered, \cite{NP} (see, \emph{e.g.}, \cite[Lemma VI.8.3]{DGZ}). For further information and historical background on Asplund spaces we refer to \cite[\S~1.5]{DGZ}, \cite{Fabian book}, \cite{Phelps book}, and the references therein.

\smallskip

A Banach space is \emph{weakly compactly generated} (WCG, for short) if it admits a linearly dense, weakly compact subset \cite{AmirLind}; typical examples of WCG Banach spaces are separable Banach spaces, reflexive ones, $c_0(\Gamma)$ for every index set $\Gamma$, or $L_1(\mu)$, for a finite measure $\mu$. If $\X$ is WCG, then the dual ball $(B_{\X^*},w^*)$ is an Eberlein compactum, \cite{AmirLind}. The converse implication however fails, in light of the famous Rosenthal's example \cite{Rosenthal} of a non WCG subspace of a WCG Banach space. There even are WCG Banach spaces with unconditional basis that contain non WCG subspaces (with unconditional bases), \cite{ArMe}. The situation is more symmetric in the $\C(\K)$ case, since $\C(\K)$ is WCG if and only if $\K$ is Eberlein, if and only if $(B_{\C(\K)^*},w^*)$ is Eberlein, \cite{AmirLind}.\smallskip

A broader class of Banach spaces is constituted by WLD Banach spaces. A Banach space $\X$ is \emph{weakly Lindel\"of determined} (hereinafter, WLD) if the dual ball $(B_{\X^*},w^*)$ is Corson, \cite{ArMe WLD}. From the stability properties of Corson compacta under subspaces and continuous images, it readily follows that WLD spaces are closed under subspaces and quotients. As we already saw in Theorem \ref{Th: shrinking M-basis}, when restricted to Asplund spaces, WCG and WLD collapse to the same notion. Therefore, for our paper, the subtlety of the distinction between WCG, subspace of WCG and WLD will not be extremely relevant. A $\C(\K)$ space is WLD if and only if $\K$ is Corson and it has property (M), namely every measure on $\K$ has separable support, \cite{AMN}.

Incidentally, these results yield a, unnecessarily sophisticated, Banach space theoretical proof of Alster's result \cite{Alster} that scattered Corson compacta are Eberlein. Indeed, every scattered compact space has (M), since measures are even countably supported, \cite{Rudin}; hence, $\C(\K)$ is WLD and Asplund when $\K$ is Corson and scattered. But then $\C(\K)$ is WCG, whence $\K$ is Eberlein.

\section{The proof of Theorem \ref{MTh: Kr}}\label{Sec: Proof of Th B}
The goal of the present section is the proof of Theorem \ref{MTh: Kr}. As we already mentioned in the Introduction, the construction of the family $\Fr$ depends upon the choice of a $\ro$-function with some additional properties. Therefore, we start the section recalling some facts concerning $\ro$-functions.\smallskip

We will consider functions $\ro\colon [\omega_1]^2\to\omega$ and it will be convenient to identify their domain $[\omega_1]^2$ with the set
$$\{(\alpha,\beta)\in{\omega_1}^2\, \colon \alpha<\beta\}.$$
This just amounts to replacing the unordered pair $\{\alpha,\beta\}$ with the ordered one $(\alpha,\beta)$, where $\alpha<\beta$. This permits writing $\ro(\alpha,\beta)$ instead of the more cumbersome $\ro(\{\alpha,\beta\})$. It will also be useful to extend the domain of such functions by adding the boundary condition that $\ro(\alpha,\alpha)=0$ ($\alpha<\omega_1$). Throughout the section, we shall adhere to these conventions (which follow precisely those of \cite{T}).

\begin{definition} A \emph{$\ro$-function}, or \emph{ordinal metric}, on $\omega_1$ is a function $\ro\colon [\omega_1]^2 \to\omega$ with the following properties:
\begin{enumerate}
    \item[($\ro$1)] $\{\xi\leq\alpha \colon \ro(\xi,\alpha)\leq n\}$ is a finite set, for every $\alpha<\omega_1$ and $n<\omega$,
    \item[($\ro2$)] $\ro(\alpha,\gamma)\leq \max\{\ro(\alpha,\beta), \ro(\beta,\gamma)\}$ for every $\alpha<\beta<\gamma<\omega_1$,
    \item[($\ro$3)] $\ro(\alpha,\beta)\leq \max\{\ro(\alpha,\gamma), \ro(\beta,\gamma)\}$ for every $\alpha<\beta<\gamma<\omega_1$.
\end{enumerate}
\end{definition}

Several functions with the above properties, frequently originating as characteristics of some walk on ordinals, are constructed and studied in \cite[Chapter 3]{T}. Here we shall need the existence of a $\ro$-function on $\omega_1$ with the following properties.

\begin{proposition}[{\cite[Lemma 3.2.2]{T}}]\label{Prop: Our ro function} There exists a $\ro$-function $\ro\colon [\omega_1]^2 \to\omega$ such that
\begin{romanenumerate}
    \item $\ro(\alpha,\beta)>0$ for all $\alpha<\beta<\omega_1$,
    \item $\ro(\alpha,\gamma)\neq\ro(\beta,\gamma)$, for all $\alpha<\beta<\gamma<\omega_1$.
\end{romanenumerate}
\end{proposition}

Let us just mention that condition (i) is obvious from the definition of $\ro$ (given in \cite[Definition 3.2.1]{T}), while (ii) is the content of \cite[Lemma 3.2.2]{T}. \smallskip

We are now in position to define the desired compact space $\Kr$.

\begin{definition} Let $\ro\colon [\omega_1]^2\to \omega$ be a $\ro$-function satisfying (i)-(ii) of Proposition \ref{Prop: Our ro function}. Given $n<\omega$ and $\alpha<\omega_1$, let
$$F_n(\alpha):=\{\xi\leq\alpha \colon \ro(\xi,\alpha)\leq n\}.$$
Moreover, we denote
$$\Fr:=\{F_n(\alpha)\colon n<\omega,\, \alpha<\omega_1\} \qquad\text{and}\qquad \Kr:=\overline{\Fr},$$
where the closure is intended in the pointwise topology of $\P(\omega_1)$.
\end{definition}

Let us list in the next fact some properties of the sets $F_n(\alpha)$ that are immediate consequence of their definition.
\begin{fact}\label{Fact: prop of F_n(alpha)} For every $\alpha<\omega_1$ the following hold:
\begin{romanenumerate}
    \item $F_n(\alpha)\subseteq [0,\alpha]$ for all $n<\omega$,
    \item $F_0(\alpha)=\{\alpha\}$,
    \item $F_k(\alpha)\subseteq F_n(\alpha)$ for $k\leq n<\omega$,
    \item $|F_n(\alpha)|\leq n+1$ for every $n<\omega$,
    \item the sequence $(F_n(\alpha))_{n<\omega}$ converges to $[0,\alpha]$.
\end{romanenumerate}
\end{fact}
Concerning the proof of these assertions, we just note that, according to Proposition \ref{Prop: Our ro function}(ii), $\ro(\cdot,\alpha)$ defines an injection of $[0,\alpha]$ into $\omega$; whence (iv) follows. Let us also observe that (iv) is a uniform version of condition ($\ro$1). \smallskip

We are now ready for the main result of the section, namely the proof that the compact space $\Kr$ defined above satisfies the conditions stated in Theorem \ref{MTh: Kr}. For convenience of the reader, we shall repeat here the statement of the result under consideration.

\begin{theorem} The compact space $\Kr$ defined above has the following properties:
\begin{romanenumerate}
\item $\{\alpha\}\in\Kr$ for every $\alpha<\omega_1$,
\item $[0,\alpha)\in\Kr$ for every $\alpha\leq\omega_1$,
\item if $A\in\Kr$ is an infinite set, then $A=[0,\alpha)$ for some $\alpha\leq\omega_1$,
\item $\Kr$ is scattered.
\end{romanenumerate}
\end{theorem}

Due to our identification between $\P(\omega_1)$ and $\{0,1\}^{\omega_1}$, we can see $\Fr$ as a subset of $\Sigma(\omega_1)$. Since $\Fr$ is dense in $\Kr$, we derive that $\Kr$ is Valdivia. We will see later in Section \ref{Sec: semi-Eberlein} that it is even semi-Eberlein.

As the reader will see, the main part of the proof, where we use the properties of the function $\ro$, consists in establishing (iii). Then (iv) is consequence of (iii), while (i) and (ii) are immediate.
\begin{proof} (i) is obvious, since $\{\alpha\}=F_0(\alpha)\in\Fr$ for every $\alpha<\omega_1$, in light of Fact \ref{Fact: prop of F_n(alpha)}(ii).\smallskip

(ii) Pick $\alpha\leq\omega_1$ arbitrarily. If $\alpha=\alpha'+1$ is a successor ordinal, then $[0,\alpha)=[0,\alpha']=\lim_{n<\omega}F_n(\alpha')$, by Fact \ref{Fact: prop of F_n(alpha)}(v). Thus, $[0,\alpha)\in\Kr$. If $\alpha>0$ is limit, then $[0,\alpha)$ is the limit of the net $([0,\beta+1))_{\beta<\alpha}$, whence $[0,\alpha)\in\Kr$ in this case as well. Finally, if $\alpha=0$, $[0,\alpha)=\emptyset=\lim_{j<\omega}\{j\}\in\Kr$.\smallskip

(iii) Let $A\in\Kr$ be an infinite set. Our goal is to show that
$$\alpha\in A,\, \tilde{\alpha}<\alpha \implies \tilde{\alpha} \in A.$$
We shall start with the case when $A$ is a countable set.

\begin{claim}\label{Claim: countable is initial interval} Let $A\in\Kr$ be such that $|A|=\omega$. If $\alpha\in A$ and $\tilde{\alpha}<\alpha$, then $\tilde{\alpha}\in A$.
\end{claim}

\begin{proof}[Proof of Claim \ref{Claim: countable is initial interval}]\renewcommand\qedsymbol{$\square$} Pick $A\in\Kr$ such that $|A|=\omega$ and fix $\alpha\in A$ and $\tilde{\alpha} <\alpha$. Since $A$ is countable, it belongs to the $\Sigma$-subspace $\Sigma(\omega_1)$, which is Fr\'echet--Urysohn. Therefore, there exists a sequence $(F_{n_k}(\alpha_k))_{k<\omega}\subseteq \Fr$ such that $F_{n_k}(\alpha_k)\to A$. Observe that, $A$ being infinite, it cannot belong to $\Fr$; thus, we can select the sequence $(F_{n_k}(\alpha_k))_{k<\omega}$ to be injective. Moreover, up to discarding finitely many terms from the sequence, we can also assume that $\alpha\in F_{n_k}(\alpha_k)$ for each $k<\omega$ (since $\alpha\in A$ and $F_{n_k}(\alpha_k)\to A$).

We next observe that the sequence $(n_k)_{k<\omega}$ is necessarily unbounded. Indeed, if this were not the case, by Fact \ref{Fact: prop of F_n(alpha)}(iv) there would exist $M<\omega$ such that $|F_{n_k}(\alpha_k)|\leq M$ for each $k<\omega$. But then, it would follow that $|A|\leq M$ as well, contrary to our assumption. Consequently, up to passing to a subsequence and relabelling, we can also assume that $\ro(\tilde{\alpha},\alpha)\leq n_k$ for each $k<\omega$.

The condition $\alpha\in F_{n_k}(\alpha_k)$ yields that $\alpha\leq\alpha_k$ and $\ro(\alpha,\alpha_k)\leq n_k$. Therefore, $\tilde{\alpha}<\alpha\leq\alpha_k$ and property ($\ro2$) imply that
$$\ro(\tilde{\alpha},\alpha_k)\leq \max\{ \ro(\tilde{\alpha},\alpha), \ro(\alpha,\alpha_k)\}\leq n_k.$$
Consequently, we obtain $\tilde{\alpha}\in F_{n_k}(\alpha_k)$ for every $k<\omega$, which implies $\tilde{\alpha}\in A$, as we desired.
\end{proof}

Finally, we shall consider the case when $A$ is uncountable and we shall show that $A=[0,\omega_1)$. Towards a contradiction, assume that there exists $\alpha<\omega_1$ such that $\alpha\notin A$. Since $A$ is uncountable, we can also pick $\beta\in A$, $\beta>\alpha$, such that $A\cap[0,\beta]$ is an infinite set (which might not belong to $\Kr$, though). However, in light of Lemma \ref{Lemma: closing-off}, there exists $\gamma<\omega_1$, $\gamma>\beta$, such that $A\cap[0,\gamma)\in\Kr$. Therefore, $A\cap[0,\gamma)\in\Kr$ is a set of cardinality $\omega$ that is not an initial interval, since $\alpha\notin A\cap[0,\gamma)$, while $\beta\in A\cap[0,\gamma)$. This contradicts Claim \ref{Claim: countable is initial interval} above and concludes the proof of (iii). \smallskip

(iv) Let $\D$ be any closed subset of $\Kr$. We shall show that $\D$ admits an isolated point. In the case that every element of $\D$ is an initial interval (namely, of the form $[0,\alpha)$, for some $\alpha< \omega_1$), then $\D$ is homeomorphic to a closed subset of $[0,\omega_1]$ and, therefore, it contains an isolated point. Consequently, we can assume that there is $D_0\in\D$ that is not an initial interval. Thus, we can pick $\alpha<\beta<\omega_1$ such that $\alpha\notin D_0$ and $\beta\in D_0$.

Consider the partially ordered set
$$\mathscr{P}:=\{D\in\D\colon \alpha\notin D, D_0\subseteq D\},$$
partially ordered by inclusion. Every chain in such a partially ordered set admits an upper bound given by the union of its elements; indeed, such union belongs to $\D$, $\D$ being a closed subset of $\Kr$. By Zorn's lemma, we can pick a maximal element $M\in \mathscr{P}$. Since $\alpha\notin M$ and $\beta\in M$, $M$ is not an initial interval; hence, according to (iii), it is a finite set. Consequently, the set
$$\mathcal{U}:=\{D\in\D\colon \alpha\notin D, M\subseteq D\},$$
is an open neighbourhood of $M$ in $\D$. However, by maximality of $M$, $\mathcal{U}=\{M\}$, which shows that $M\in\D$ is the desired isolated point.
\end{proof}

\begin{remark} The assertion, in (iii), that the unique uncountable set in $\Kr$ is $[0,\omega_1)$ can also be proved in a different way using, instead of Lemma \ref{Lemma: closing-off}, Kalenda's extension \cite{Kalenda embed} of Deville--Godefroy's theorem \cite{DG} that we mentioned already in Section \ref{Sec: preliminaries}. Indeed, as $|A|=\omega_1$, there exists a continuous injection $\p\colon [0,\omega_1] \to\Kr$ such that $\p(\omega_1)=A$, $\p(\alpha)\subseteq \p(\beta)$ whenever $\alpha<\beta\leq \omega_1$ and $|\p(\alpha)|\leq\max\{|\alpha|, \omega\}$. If we set $A_\alpha:=\p(\alpha)$ $(\alpha<\omega_1)$, $(A_\alpha)_{\alpha<\omega_1}$ is a non-decreasing family of sets such that $\cup_{\alpha<\omega_1}A_\alpha=A$ (due to the continuity of $\p$) and $|A_\alpha|\leq\omega$. Therefore, there exists $\alpha_0<\omega_1$ such that $A_\alpha$ is an infinite set, whenever $\alpha_0\leq\alpha <\omega_1$. By Claim \ref{Claim: countable is initial interval}, every such $A_\alpha$ is an initial interval, whence $A$ is an initial interval as well.
\end{remark}

\section{Theorem \ref{MTh: Xr} and other consequences of Theorem \ref{MTh: Kr}}\label{Sec: Consequences of Th B}
\subsection{Proof of Theorem \ref{MTh: Xr}}\label{Sec: Proof Th A} 
This section is dedicated to the proof of our main result.

\begin{proof}[Proof of Theorem \ref{MTh: Xr}] According to Theorem \ref{MTh: Kr}, we can pick a family $\Fr \subseteq[\omega_1]^{<\omega}$ such that the compact $\Kr:=\overline{\Fr}\subseteq\P (\omega_1)$ has the following properties:
\begin{romanenumerate}
\item $\{\alpha\}\in\Kr$ for every $\alpha<\omega_1$,
\item $[0,\alpha)\in\Kr$ for every $\alpha\leq\omega_1$,
\item if $A\in\Kr$ is an infinite set, then $A=[0,\alpha)$ for some $\alpha\leq\omega_1$,
\item $\Kr$ is scattered.
\end{romanenumerate}

We define a biorthogonal system $\{f_\gamma; \mu_\gamma\}_{ \gamma<\omega_1}$ in the Banach space $\C(\Kr)$ as follows. For $\gamma<\omega_1$, let

\begin{equation*}\begin{split} f_\gamma \in \C(\Kr) \qquad \qquad & f_\gamma(A)=\begin{cases} 1 & \gamma\in A \\ 0 & \gamma \notin A \end{cases} \qquad (A\in \Kr) \\
\mu_\gamma:=\delta_{\{\gamma\}} \in \M(\Kr) \qquad\qquad & \mu_\gamma(S)=\begin{cases} 1 & \{\gamma\} \in S \\ 0 & \{\gamma\} \notin S \end{cases} \qquad (S\subseteq \Kr).
\end{split}\end{equation*}
Note that, by (i), $\{\gamma\}\in\Kr$ for each $\gamma<\omega_1$, whence each $\mu_\gamma$ is, indeed, a measure on $\Kr$. Moreover, $f_\gamma$ is a continuous function on $\Kr$, since the set $\{A\in\Kr \colon \gamma\in A\}$ is clearly clopen. Finally, $\langle\mu_\gamma, f_\alpha \rangle= \langle\delta_{\{\gamma\}}, f_\alpha \rangle= f_\alpha(\{\gamma\})=\delta_{\alpha,\gamma}$. Therefore, $\{f_\gamma; \mu_\gamma\}_{ \gamma<\omega_1}$ is a well-defined biorthogonal system in $\C(\Kr)$.\smallskip

We are now in position to define the Banach space $\Xr:=\overline{\rm span} \{f_\gamma\}_{\gamma<\omega_1}\subseteq \C(\Kr)$. We shall show that $\Xr$ is the Banach space we are seeking, namely that $\Xr$ is an Asplund space with a $1$-norming M-basis and that $\Xr$ is not WLD. Some of these properties are actually obvious. Indeed, $\Xr$ is an Asplund space, being a subspace of the Asplund space $\C(\Kr)$, by (iv). Moreover, the biorthogonal system $\{f_\gamma;\mu_\gamma\}_{ \gamma<\omega_1}$ naturally induces a biorthogonal system $\{f_\gamma; \mu_\gamma\cut_\Xr\}_{\gamma<\omega_1}$ on $\Xr$. Such a system is clearly a fundamental (in the sense that $\{f_\gamma\}_{ \gamma<\omega_1}$ is linearly dense in $\Xr$) biorthogonal system.

\begin{claim}\label{claim: not WLD} $\Xr$ is not WLD.
\end{claim}
\begin{proof}[Proof of Claim \ref{claim: not WLD}] \renewcommand\qedsymbol{$\square$} We shall show that the dual ball $(B_{\Xr^*},w^*)$ is not Corson, by proving that $[0,\omega_1]$ embeds therein. From condition (iii) we infer that $\beta\mapsto[0,\beta)$ defines an embedding $\iota$ of $[0,\omega_1]$ into $\Kr$; therefore, it suffices to prove that $(B_{\Xr^*}, w^*)$ contains a homeomorphic copy of $\Kr$. This is actually a standard consequence of the fact that the functions $\{f_\gamma\}_{ \gamma<\omega_1}$ separate points of $\Kr$. Let us give the details below.

It is well known that every compact space $\K$ embeds in $(B_{\M(\K)},w^*)$ via the map $\delta\colon \K\to(B_{\M(\K)},w^*)$ given by $p\mapsto \delta_p$ ($p\in\K$). Moreover, we can consider the $w^*$-$w^*$-continuous quotient map $q\colon \M(\Kr)\to\Xr^*$ defined by $\mu\mapsto \mu\cut_\Xr$. Hence, the function $e:=q\circ\delta\colon \Kr\to (B_{\Xr^*},w^*)$, namely the function given by the rule $A\mapsto \delta_A\cut_\Xr$ ($A\in\Kr$), is continuous. 

$$\xymatrix{ [0,\omega_1]\; \ar@{^{(}->}[r]^\iota & \;\Kr\; \ar@{^{(}->}[rr]^(.35)\delta \ar@{^{(}-->}@<-1.4ex>[rrd]_(.4)e && {(B_{\M(\Kr)},w^*)} \ar[d]^q \\ & && \;\;(B_{\Xr^* },w^*)}$$

We shall show that $e$ is injective, which, due to the compactness of $\Kr$, implies that it is the desired homeomorphic embedding. Given distinct $A,B\in\Kr$, pick $\gamma\in A\Delta B$; assume, for example, $\gamma\in A\setminus B$. Then $\langle\delta_A\cut_\Xr, f_\gamma \rangle= \langle\delta_A, f_\gamma\rangle= f_\gamma(A)=1$, while $\langle\delta_B\cut_\Xr, f_\gamma \rangle= f_\gamma (B)=0$. Hence, $\delta_A\cut_\Xr\neq \delta_B\cut_\Xr$, as desired.
\end{proof}

In order to conclude the proof we thus just need to prove that the biorthogonal system $\{f_\gamma; \mu_\gamma\cut_\Xr \}_{\gamma<\omega_1}$ is $1$-norming for $\Xr$. Note that this implies, in particular, that ${\rm span}\{\mu_\gamma\cut_\Xr\} _{\gamma<\omega_1}$ is $w^*$-dense in $\Xr^*$. We start with the following observation.

\begin{claim}\label{Claim: sum of Diracs} Let $A\in\Fr$. Then
\begin{equation}\label{Eq: sum of Diracs}
    \delta_A\cut_\Xr = \sum_{\alpha\in A} \delta_{\{\alpha\}}\cut_\Xr.
\end{equation}
\end{claim}
In particular, it follows that
\begin{equation}\label{Eq: Diracs in span}
    \{\delta_A\cut_\Xr \colon A\in\Fr\}\subseteq {\rm span}\{\mu_\gamma\cut_\Xr\} _{\gamma<\omega_1}.
\end{equation}

\begin{proof}[Proof of Claim \ref{Claim: sum of Diracs}] \renewcommand\qedsymbol{$\square$} Recall that every element $A\in\Fr$ is a finite set and $\{\alpha\}\in\Fr$, for every $\alpha\in A$; thus, the right hand side in (\ref{Eq: sum of Diracs}) is a well-defined functional on $\Xr$. Of course, it is sufficient to show that the functional $\delta_A\cut_\Xr - \sum_{\alpha\in A} \delta_{\{\alpha\}}\cut_\Xr$ vanishes on the linearly dense set $\{f_\gamma\}_{\gamma<\omega_1}$. Fix $\gamma<\omega_1$; by definition of $f_\gamma$, we have
$$\left\langle\sum_{\alpha\in A}\delta_{\{\alpha\}}, f_\gamma\right \rangle= \sum_{\alpha\in A}f_\gamma(\{\alpha\})= \sum_{\alpha\in A}\delta_{\gamma,\alpha}= \begin{cases} 1 & \gamma\in A \\ 0 & \gamma \notin A \end{cases}= f_\gamma(A)= \langle\delta_A, f_\gamma\rangle.$$
\end{proof}

Finally, for every $f\in\Xr$ we have
$$\|f\|=\max_{A\in\Kr}|f(A)|= \sup_{A\in\Fr}|f(A)|= \sup_{A\in\Fr}|\langle\delta_A ,f\rangle|$$
$$\stackrel{(\ref{Eq: Diracs in span})} {\leq} \sup\left\{|\langle \mu,f\rangle|\colon \mu\in {\rm span}\{\mu_\gamma\cut_\Xr\}_{\gamma<\omega_1} , \|\mu\|\leq1 \right\}.$$
This shows that the M-basis $\{f_\gamma; \mu_\gamma\cut_\Xr \}_{\gamma<\omega_1}$ is $1$-norming for $\Xr$ and concludes the proof.
\end{proof}

\begin{remark}\label{Rmk: M-basis is Auerbach} As it is clear from the proof, the M-basis $\{f_\gamma; \mu_\gamma\cut_\Xr \}_{\gamma<\omega_1}$ satisfies $\|f_\gamma\|=\|\mu_\gamma\cut_\Xr\|=1$ for each $\gamma<\omega_1$; in other words, it is additionally an Auerbach basis.
\end{remark}

\subsection{Further properties of the space \texorpdfstring{$\Xr$}{X ro}}
In this section we shall prove the result mentioned in the Introduction that the Banach space $\Xr$ can also be assumed to have a long monotone Schauder basis. The argument will be a simple modification of the construction in Section~\ref{Sec: Proof Th A}; in particular, we shall continue to denote $\{f_\gamma; \mu_\gamma\}_{\gamma<\omega_1}$ the biorthogonal system introduced there.

The main idea is that the Banach space $\Xl:=\overline{\rm span}\{f_\gamma\}_{\gamma\in\Lambda}$ shares the main features of $\Xr$, whenever $\Lambda$ is an uncountable subset of $\omega_1$. On the other hand, it is a folklore result that if $\{f_\gamma; \mu_\gamma\}_{\gamma<\omega_1}$ is a $1$-norming M-basis for a Banach space $\X$, then there exists an uncountable subset $\Lambda$ of $\omega_1$ such that $\{f_\gamma\}_{\gamma\in\Lambda}$ is a monotone long Schauder basis (see Fact~\ref{Fact: norming gives Schauder basis} below).\smallskip

For the definitions of long Schauder bases, long (Schauder) basic sequences and their basis constants we shall refer, \emph{e.g.}, to \cite[\S~4.1]{HMVZ}, or \cite[\S~17]{S2}. Here, we shall restrict ourselves to spelling out the following useful characterisation, \cite[Theorem 17.8]{S2}. A collection $(e_\gamma)_{\gamma< \Gamma}$ of non-zero vectors in a Banach space $\X$ is a long basic sequence if and only if there exists a constant $C\geq 1$ such that
$$\|y\|\leq C\|y+z\|$$
whenever $y\in \overline{\rm span}\{e_\gamma\}_{\gamma<\Omega}$, $z\in \overline{\rm span}\{e_\gamma\}_{\Omega\leq\gamma<\Gamma}$, and $\Omega<\Gamma$ (and, in this case, the basis constant of $(e_\gamma)_{\gamma< \Gamma}$ is at most $C$).\smallskip

We are now in position to state the following folklore fact, based on Mazur's technique (compare with \cite[Corollary 4.11]{HMVZ}).

\begin{fact}\label{Fact: norming gives Schauder basis} Let $\{f_\gamma; \mu_\gamma\}_{\gamma<\omega_1}$ be a $\theta$-norming M-basis for a Banach space $\X$. Then there exists an uncountable subset $\Lambda$ of $\omega_1$ such that $\{f_\gamma\}_{\gamma\in \Lambda}$ is a long basic sequence in $\X$ (in the natural ordering induced on $\Lambda$ by $\omega_1$) with basis constant at most $1/\theta$.
\end{fact}

\begin{proof} If $\Omega$ is any countable subset of $\omega_1$, there is a countable subset $S_\Omega$ of $\omega_1$ such that
$$\theta \|x\|\leq \sup\left\{|\langle\mu,x \rangle|\colon \mu\in {\rm span}\{\mu_\gamma\} _{\gamma\in S_\Omega}, \|\mu\|\leq 1 \right\},$$
for every $x\in \overline{\rm span}\{f_\gamma\} _{\gamma\in\Omega}$. Indeed, let $(g_j)_{j<\omega}$ be a dense sequence in $\overline{\rm span}\{f_\gamma\} _{\gamma\in\Omega}$ and, for each $j<\omega$, find a sequence $(\mu_k^j)_{k<\omega}\in {\rm span}\{\mu_\gamma\}_{\gamma<\omega_1}$, $\|\mu_k^j\|\leq1$, such that 
$$\theta \|g_j\|\leq \sup_{k<\omega}|\langle\mu_k^j, g_j \rangle|.$$
Then, any countable set $S_\Omega\subseteq\omega_1$ such that $(\mu_k^j)_{k,j<\omega}\subseteq {\rm span}\{\mu_\gamma\} _{\gamma\in S_\Omega}$ is as desired.\smallskip

Therefore, a standard transfinite induction argument gives an uncountable subset $\Lambda=(\lambda_\xi)_{\xi<\omega_1}$ of $\omega_1$, where the ordinals $\lambda_\xi$ are enumerated in increasing order, and an increasing family $(S_\xi)_{\xi<\omega_1}$ of countable subsets of $\omega_1$ such that $S_\xi < \lambda_\xi$ and
$$\theta \|x\|\leq \sup\left\{|\langle\mu,x \rangle|\colon \mu\in {\rm span}\{\mu_\gamma\} _{\gamma\in S_\xi}, \|\mu\|\leq 1 \right\},$$
whenever $x\in \overline{\rm span}\{f_{\lambda_\gamma}\} _{\gamma<\xi}$.

Consequently, for every $\xi<\omega_1$ and every $x\in\overline{\rm span}\{f_{\lambda_\gamma}\}_{\gamma<\xi}$ and $y\in \overline{\rm span}\{f_{\lambda_\gamma}\}_{\xi\leq\gamma<\omega_1}$ we have $y\in \ker \mu_\gamma$ ($\gamma\in S_\xi$); hence,
\begin{eqnarray*}
\theta \|x\| &\leq& \sup\left\{|\langle\mu,x \rangle|\colon \mu\in {\rm span}\{\mu_\gamma\} _{\gamma\in S_\xi}, \|\mu\|\leq 1 \right\} \\&=&  \sup\left\{|\langle\mu,x+y \rangle|\colon \mu\in {\rm span}\{\mu_\gamma\} _{\gamma\in S_\xi}, \|\mu\|\leq 1 \right\} \\
&\leq& \|x+y\|.
\end{eqnarray*}
By the characterisation mentioned above, we derive that $(f_\gamma)_{\gamma\in\Lambda}$ is the desired long basic sequence with basis constant at most $1/\theta$.
\end{proof}

\begin{remark} We stated the result for M-bases of length $\omega_1$ since we shall only need this particular case; however, a standard modification of the argument also proves the following more general facts. If $\Gamma$ is an uncountable ordinal and $\{f_\gamma; \mu_\gamma\}_{\gamma< \Gamma}$ is a $\theta$-norming M-basis for $\X$, there exist a set $\Lambda\subseteq\Gamma$ with $|\Lambda|=|\Gamma|$ and a well ordering $\preccurlyeq$ of $\Lambda$ such that $(f_\gamma)_{\gamma\in \Lambda}$ is a long basic sequence in $\X$ in the $\preccurlyeq$ ordering. In case $\Gamma$ is a regular cardinal, the ordering can be chosen to be the one induced by $\Gamma$. Finally, if $\Gamma$ is countable, for every $\e>0$, there exists an increasing sequence $(\gamma_j)_{j<\omega}$ in $\Gamma$ such that $(f_{\gamma_j})_{j<\omega}$ is a basic sequence with basis constant at most $1/\theta +\e$.
\end{remark}

Next, we shall observe the obvious fact that passing to a subset of the index set of a norming M-basis produces a norming M-basis for the corresponding subspace.
\begin{fact}\label{Fact: norming remains to subsets} Let $\{f_\gamma; \mu_\gamma\}_{\gamma<\Gamma}$ be a $\theta$-norming M-basis for a Banach space $\X$ and let $\Lambda\subseteq \Gamma$. Set $\Xl:=\overline{\rm span}\{f_\gamma \}_{\gamma\in\Lambda}$. Then $\{f_\gamma; \mu_\gamma\cut_\Xl\} _{\gamma\in\Lambda}$ is a $\theta$-norming M-basis for $\Xl$.
\end{fact}

\begin{proof} Obviously, $\{f_\gamma; \mu_\gamma\cut_\Xl\} _{\gamma\in\Lambda}$ is a fundamental biorthogonal system for $\Xl$. If $\mu\in {\rm span}\{\mu_\gamma\}_{\gamma<\Gamma}$, write $\mu=\sum_{\gamma<\Gamma} a_\gamma\mu_\gamma$, where only finitely many scalars $a_\gamma$ are non-zero. Then
$$\mu\cut_\Xl=\sum_{\gamma\in\Lambda}a_\gamma\mu_\gamma\cut_\Xl.$$
Therefore, when $x\in\Xl$, we have
\begin{eqnarray*} \theta\|x\| &\leq& \sup\left\{|\langle\mu\cut_\Xl,x\rangle|\colon \mu\in {\rm span}\{\mu_\gamma\}_{\gamma<\Gamma}, \|\mu\|\leq1 \right\} \\
&=& \sup\left\{|\langle\mu,x\rangle|\colon \mu\in {\rm span}\{\mu_\gamma\cut_\Xl\}_{\gamma\in\Lambda}, \|\mu\|\leq1 \right\}.
\end{eqnarray*}
Thus, ${\rm span}\{\mu_\gamma\cut_\Xl\}_{\gamma\in\Lambda}$ is $\theta$-norming for $\Xl$, as desired.
\end{proof}

We are now ready to state and prove the announced result.
\begin{theorem}\label{Th: Xr with basis} There exists an Asplund space $\X$ with a $1$-norming Auerbach basis $\{f_\gamma;\mu_\gamma\}_ {\gamma<\omega_1}$ such that $(f_\gamma)_{\gamma<\omega_1}$ is a long monotone Schauder basis and yet $\X$ is not WCG.
\end{theorem}

\begin{proof} Consider the Banach space $\Xr$ with $1$-norming M-basis $\{f_\gamma;\mu_\gamma\}_ {\gamma<\omega_1}$ constructed in Section \ref{Sec: Proof Th A}. Recall that $\|f_\gamma\|=\|\mu_\gamma\|=1$ ($\gamma<\omega_1$), see Remark \ref{Rmk: M-basis is Auerbach}. In the light of Fact \ref{Fact: norming gives Schauder basis}, there exists an uncountable subset $\Lambda$ of $\omega_1$ such that $(f_\gamma)_ {\gamma\in\Lambda}$ is a monotone long Schauder basis for $\Xl:=\overline{\rm span}\{f_\gamma\}_{\gamma\in \Lambda}$. Moreover, Fact \ref{Fact: norming remains to subsets} yields that $\Xl$ admits a $1$-norming Auerbach basis, given by $\{f_\gamma; \mu_\gamma\cut_\Xl\} _{\gamma\in\Lambda}$. Since $\Xl\subseteq \Xr$, it is also clear that $\Xl$ is Asplund. Consequently, it only remains to prove that the space $\X_{\Lambda}$ is not WCG. 

In order to achieve that, we shall show that $[0,\omega_1]$ embeds into $(B_{\Xl^*},w^*)$; hence, $(B_{\Xl^*},w^*)$ is not Corson and $\Xl$ is not WLD. Let us enumerate $\Lambda$ as an increasing transfinite sequence $(\lambda_{\xi})_{\xi< \omega_1}$. We then consider the continuous function $\pi\colon[0,\omega_1]\to (B_{\Xl^*},w^*)$ defined by $\pi(\alpha):=\delta_{[0, \alpha)}\cut_\Xl$; observe that for $\alpha<\omega_1$ and $\lambda\in\Lambda$ we have $\langle \pi(\alpha), f_{\lambda}\rangle=\langle \delta_{[0,\alpha)}\cut_\Xl, f_{\lambda}\rangle= f_{\lambda}([0,\alpha))$. Then, using the density of ${\rm span}\{f_{\lambda}\}_{\lambda\in\Lambda}$ in $\X_{\Lambda}$, we obtain:
\begin{romanenumerate}
    \item $\pi(\alpha)=\pi(\lambda_{\xi+1})$ if $\alpha\in (\lambda_{\xi},\lambda_{\xi+1}]$;
    \item $\pi(\alpha)=\pi(\lambda_0)$ if $\alpha\in[0,\lambda_0]$.
    \item\label{Eq:  limit} $\pi(\alpha)=\pi(\lambda_\xi)$ if $\xi$ is a limit ordinal and $\alpha\in[\sup_{\beta<\xi}\lambda_{\beta},\lambda_{\xi}]$.
\end{romanenumerate}
We observe that, if $\xi_1< \xi_2<\omega_1$, then $\langle \pi(\lambda_{\xi_1}), f_{\lambda_{\xi_1}}\rangle= f_{\lambda_{\xi_1}}([0, \lambda_{\xi_1}))=0$, while $\langle \pi(\lambda_{\xi_2}), f_{\lambda_{\xi_1}}\rangle= f_{\lambda_{\xi_1}}([0, \lambda_{\xi_2}))=1$. Therefore, $\pi(\lambda_{\xi_1})\neq \pi(\lambda_{\xi_2})$.  Consequently, the map $h\colon [0,\omega_1]\to (B_{\Xl^*},w^*)$ defined by 
$$h(\xi)=\begin{cases}
\pi(\lambda_{\xi}), & \xi<\omega_1,\\
\pi(\omega_1), & \xi=\omega_1,
\end{cases}$$
is a injection into $(B_{\Xl^*},w^*)$. Let us show that $h$ is also continuous. Indeed, let $\gamma\in [0,\omega_1]$ be a limit ordinal and let $\{\gamma_{\eta}\}$ be a net converging to $\gamma=\sup \gamma_{\eta}$. In case $\gamma<\omega_1$, then by \eqref{Eq:  limit} we have $\pi(\sup \lambda_{\gamma_{\eta}})=\pi(\lambda_{\gamma})$, which, by the continuity of $\pi$, yields $\lim h(\gamma_{\eta})=h(\gamma)$. On the other hand, when $\gamma=\omega_1$, the continuity of the map $\pi$ ensures that $\lim h(\gamma_\eta)=h(\omega_1)$. Therefore $[0,\omega_1]$ embeds into $(B_{\Xl^*},w^*)$.
\end{proof}

Answering a question of Argyros, it was proved in \cite{LT} that there exists a $c_0$-saturated, non-separable Banach space $\mathfrak{X}$ that contains no unconditional long basic sequence. In particular, every infinite-dimensional subspace of $\mathfrak{X}$ contains an unconditional basic sequence. Thus, the Banach space $\mathfrak{X}$ exhibits a radical discrepancy between the behaviour of separable and non-separable subspaces. The argument in \cite{LT} also heavily uses the machinery of $\ro$-functions---differently from how it is done in our proof---combined with techniques originating from Schlumprecht's construction of an arbitrary distortable Banach space, \cite{S}. Let us also refer to \cite{ALT}, \cite[Chapter~A.6]{AT}, and \cite[\S~3.5]{T} for a related construction.\smallskip

We do not know if the Banach space $\Xr$ is also a solution to Argyros' question, namely, we do not know if $\Xr$ contains unconditional long basic sequences. (Notice that $\Xr$ is $c_0$-saturated, being a subspace of $\C(\Kr)$, where $\Kr$ is scattered; see, \emph{e.g.}, \cite[Theorem~14.26]{FHHMZ}.) In particular, we do not know if $c_0(\omega_1)$ embeds in $\Xr$. However, we shall show the weaker fact that no uncountable subset of the vectors of the M-basis $\{f_\gamma;\mu_\gamma\}_ {\gamma<\omega_1}$ can be an unconditional long basic sequence in $\Xr$.

\begin{proposition} Let $\{f_\gamma;\mu_\gamma\}_{\gamma< \omega_1}$ be the $1$-norming M-basis for the Banach space $\Xr$. If $\Lambda\subseteq \omega_1$ is uncountable, then $(f_\gamma)_{\gamma\in\Lambda}$ is not an unconditional long basic sequence.
\end{proposition}
\begin{proof} Towards a contradiction, assume that there exists an uncountable subset $\Lambda$ of $\omega_1$ such that $(f_\gamma)_{\gamma\in\Lambda}$ is unconditional. Then, the Asplund space $\Xl:=\overline{\rm span}\{f_\gamma\}_{\gamma\in \Lambda}$ admits a long unconditional basis, whence it follows that $\{f_\gamma; \mu_\gamma\cut_\Xl\}_{\gamma\in\Lambda}$ is a shrinking M-basis (\cite{J}, \emph{e.g.}, \cite[Theorem 7.39]{HMVZ}). Consequently, it follows that $\Xl$ is WCG. However, this is not the case, as we saw in the proof of Theorem \ref{Th: Xr with basis}.
\end{proof}

\subsection{Weak P-points in semi-Eberlein compacta}\label{Sec: semi-Eberlein}
In this part we shall observe that the compact space constructed in Theorem \ref{MTh: Kr} also provides an interesting example in the theory of semi-Eberlein compact spaces. Semi-Eberlein compacta were introduced by Kubi\'s and Leiderman in \cite{KL}, as a natural weakening of the definition of Eberlein compact; further results on semi-Eberlein compacta can be found in the recent papers \cite{CCS} and \cite{CRS}. More precisely, the definition of a semi-Eberlein compact space originates from the definition of Eberlein compact by the same generalisation that leads to Valdivia compacta from Corson ones. The formal definition reads as follows.

\begin{definition} A compact space is \emph{semi-Eberlein} if it is homeomorphic to a compact space $\K \subseteq [0,1]^\Gamma$ such that $\K \cap c_0(\Gamma)$ is dense in $\K$.
\end{definition}

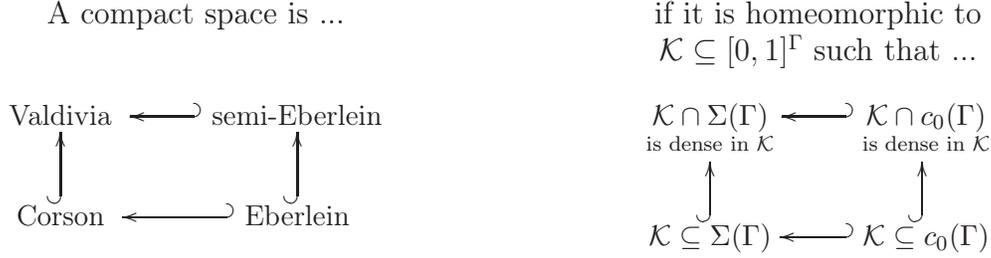
\begin{multicols}{2}\begin{center}
    A compact space is ...
    
    \columnbreak
    if it is homeomorphic to\\ $\K\subseteq [0,1]^\Gamma$ such that ...
    \end{center}
\end{multicols}
\vspace{-3em}
\begin{multicols}{2}\begin{center}
    $${\small \xymatrix{ \text{ Valdivia } & \text{ semi-Eberlein } \ar@{_{(}->}[l]\\
    \text{ Corson } \ar@{^{(}->}[u] & \text{ Eberlein } \ar@{^{(}->}[u]\ar@{_{(}->}[l] }}$$
    \vspace{.3em}
    
    \columnbreak
    $${\small \xymatrix{ \underset{\text{is dense in }\K}{\K\cap \Sigma(\Gamma)}  & \; \underset{\text{is dense in }\K}{\K\cap c_0(\Gamma)} \ar@{_{(}->}[l] \\ 
    \K\subseteq \Sigma(\Gamma) \ar@{^{(}->}[u] & \;\K\subseteq c_0(\Gamma) \ar@{^{(}->}[u]\ar@{_{(}->}[l] }}$$
    \end{center}
\end{multicols}

Obviously, every Eberlein compact is semi-Eberlein and every semi-Eberlein is Valdivia. The Tikhonov cube $[0,1]^{\omega_1}$ (or, more generally, every non Corson adequate compact, see Section \ref{S: adequate} below) is a typical example of a semi-Eberlein compact that is not Corson (and, in particular, not Eberlein). In order to offer an example of a Valdivia compact that is not semi-Eberlein, we need to recall the following notion.

\begin{definition}[\cite{GH}, \cite{Ku}] A point $p$ in a topological space $X$ is a \emph{P-point} if $p$ is not isolated and for every countable family $(U_j)_{j<\omega}$ of neighbourhoods of $p$, $\cap_{j<\omega}U_j$ is a neighbourhood of $p$. A point $p\in X$ is said a \textit{weak P-point} if $p$ is not isolated and it is limit point of no countable set in $X\setminus \{p\}$.
\end{definition}

An important result due to Kubi\'s and Leiderman (\cite[Theorem 4.2]{KL}) is the fact that semi-Eberlein compacta do not admit P-points. Since perhaps the simplest example of a P-point is the point $\omega_1$ in the compact space $[0,\omega_1]$, it follows that $[0,\omega_1]$ is not semi-Eberlein. The said result, combined with a forcing argument, also yields that the Corson compact constructed in \cite{T Corson} is not semi-Eberlein, \cite[Example 5.5]{KL}. These results motivated the question whether semi-Eberlein compacta can admit weak P-points, \cite[Question 6.1]{KL}. It is fairly easy to see that Theorem \ref{MTh: Kr} also provides a positive answer to the above question.

\begin{proposition} Let $\Kr$ be any compact space as in Theorem \ref{MTh: Kr}. Then $\Kr$ is a semi-Eberlein compact space and it admits a weak P-point.
\end{proposition}

\begin{proof} $\Kr$ is semi-Eberlein, as witnessed by the dense subset $\Fr$, that consists of finite subsets of $\omega_1$. Moreover, $[0,\omega_1)$ is a weak P-point in $\Kr$. Indeed, since every set in $\Kr\setminus\{[0,\omega_1)\}$ is countable, it follows that $\Kr \setminus\{[0,\omega_1)\}= \Kr\cap\Sigma(\Kr)$ is countably closed.
\end{proof}

\section{Adequate compacta and norming M-bases}\label{S: adequate}
In conclusion to our paper, we shall briefly discuss the main problem in the case of a $\C(\K)$ space, where $\K$ is an adequate compact. We note the easy fact that every scattered adequate compact is Eberlein, which, in particular, gives a positive answer to Godefroy's question in the realm of $\C(\K)$ spaces, $\K$ adequate. We also show that $\C(\K)$ has a $1$-norming M-basis, whenever $\K$ is adequate.

Recall that, given a set $S$, a family $\A\subseteq\P(S)$ is  \emph{adequate} \cite{Talagrand} if:
\begin{romanenumerate}
\item $\{x\}\in \A$ for each $x\in S$;
\item if $A\in \A$ and $B\subseteq A$, then $B\in \A$;
\item if $B\subseteq S$ is such that all finite subsets of $B$ belong to $\A$, then $B\in\A$.
\end{romanenumerate}

Every adequate family $\A\subseteq\P(S)$ is a closed subset of $\P(S)$. When an adequate family is considered as a compact space, it is customary to denote it $\KA$ and call it an \emph{adequate compact}.

Let $\KA$ be a scattered adequate compact; in order to see that $\KA$ is Eberlein, we just need to show that every $A\in\A$ is a finite set. If this is not the case and $A\in\A$ is infinite, then, by (ii), $\P(A)\subseteq\A$. However, $\P(A)$ is perfect, a contradiction. In other words, an adequate compact is scattered if and only if it is strong Eberlein (\emph{cf}. \cite[Lemma 2.53]{HMVZ}).

\begin{theorem}\label{t: adequate implies 1-norming} Let $\KA$ be an adequate compact. Then $\C(\KA)$ has a 1-norming M-basis.
\end{theorem}

\begin{proof} Let $\FA=\{A\in \A\colon |A|<\omega\}$, a dense subset of $\KA$. For $\Gamma\in \FA$, set
\begin{equation*}
    f_{\Gamma}(A)=\begin{cases}
    1 & \mbox{if }\Gamma\subseteq A,\\
    0 & \mbox{otherwise},
\end{cases}
\end{equation*}
\begin{equation*}
    \mu_{\Gamma}=\sum_{n=0}^{|\Gamma|}(-1)^{|\Gamma|-n}\sum_{\Delta\in[\Gamma]^n}\delta_{\Delta}.
\end{equation*}
We shall show that the family $\{f_{\Gamma};\mu_{\Gamma}\}_{\Gamma\in\FA}$ is the desired $1$-norming M-basis (note that $f_\Gamma$ is continuous, since $\Gamma$ is a finite set).

We start by showing that $\{f_\Gamma;\mu_\Gamma\} _{\Gamma\in\FA}$ is a biorthogonal system, that is $\langle\mu_{\Gamma_1}, f_{\Gamma_2}\rangle =1$ if $\Gamma_1= \Gamma_2$ and $\langle \mu_{\Gamma_1}, f_{\Gamma_2} \rangle=0$ elsewhere. Indeed, assume $\Gamma_1=\Gamma_2$. Then, if $n\leq |\Gamma_1|$ and $\Delta\in [\Gamma_1]^n$, we have $f_{\Gamma_2}(\Delta)\neq 0$ if and only if $n=|\Gamma_1|$ and $\Delta=\Gamma_1$. Therefore, $\langle \mu_{\Gamma_1}, f_{\Gamma_2}\rangle=f_{\Gamma_2}(\Gamma_2)=1$. By the same argument, $\langle\mu_{\Gamma_1}, f_{\Gamma_2}\rangle=0$ whenever $\Gamma_2\nsubseteq \Gamma_1$. 

On the other hand, suppose that $\Gamma_2\subseteq \Gamma_1$ and $\Gamma_1\neq \Gamma_2$.
Then, for every $n<|\Gamma_2|$ and every $\Delta\in[\Gamma_1]^n$ we have $f_{\Gamma_2}(\Delta)=0$. Moreover, when $n\geq |\Gamma_2|$, we have 
$$|\{\Delta\in [\Gamma_1]^n\colon f_{\Gamma_2}(\Delta) \neq 0\}|=|\{\Delta\in [\Gamma_1]^n\colon \Gamma_2 \subseteq \Delta\}|= {|\Gamma_1|-|\Gamma_2|\choose n-|\Gamma_2|}. $$
Therefore,
\begin{equation*}
    \begin{split}
    \langle\mu_{\Gamma_1}, f_{\Gamma_2}\rangle &=\sum_{n=|\Gamma_2|} ^{|\Gamma_1|}(-1)^{|\Gamma_1|-n}\sum_{\Delta\in[\Gamma_1]^n} f_{\Gamma_2}(\Delta)\\ &=\sum_{n=|\Gamma_2|}^{|\Gamma_1|} (-1)^{|\Gamma_1|-n}{|\Gamma_1|-|\Gamma_2|\choose n-|\Gamma_2|}\\ &=\sum_{n=0}^{|\Gamma_1|-|\Gamma_2|}(-1)^{|\Gamma_1|-|\Gamma_2|-n}{|\Gamma_1|-|\Gamma_2|\choose n}=0,  
    \end{split}
\end{equation*}
where the last equality depends on the binomial theorem (and the fact that $|\Gamma_1|-|\Gamma_2|>0$). \smallskip

Next, we show that ${\rm span}\{f_{\Gamma}\}_{\Gamma\in \FA}$ is dense in $\C(\KA)$. Since $f_\emptyset\equiv 1\in {\rm span}\{f_{\Gamma}\}_{\Gamma\in \FA}$ and $\{f_{\Gamma}\}_{\Gamma\in \FA}$ separates points of $\KA$, by the Stone--Weierstra\ss\ theorem, it is enough to prove that ${\rm span}\{f_{\Gamma}\}_{\Gamma\in \FA}$ is a subalgebra of $\C(\KA)$. For $\Gamma_1, \Gamma_2 \in \FA$, we have 
\begin{equation*}
    f_{\Gamma_1}\cdot f_{\Gamma_2}(A)=\begin{cases}
    1 & \mbox{if }\Gamma_1\cup\Gamma_2\subseteq A,\\
    0 & \mbox{otherwise}.
\end{cases}
\end{equation*}
Therefore, if $\Gamma_1\cup \Gamma_2 \in \FA$, then $f_{\Gamma_1}\cdot f_{\Gamma_2}=f_{\Gamma_1\cup\Gamma_2}\in {\rm span}\{f_{\Gamma}\}_{\Gamma\in \FA}$. On the other hand, if $\Gamma_1\cup \Gamma_2 \notin \FA$, there is no $A\in\KA$ with $\Gamma_1\cup \Gamma_2\subseteq A$. Hence, $f_{\Gamma_1}\cdot f_{\Gamma_2}\equiv 0$. In either case, $f_{\Gamma_1}\cdot f_{\Gamma_2}\in {\rm span}\{f_{\Gamma}\}_{\Gamma\in \FA}$, whence ${\rm span}\{f_{\Gamma}\}_{\Gamma\in \FA}$ is a subalgebra of $\C(\KA)$.\smallskip

Finally, we show that ${\rm span}\{\mu_{\Gamma}\} _{\Gamma\in \FA}$ is a $1$-norming subspace, namely (by the Hahn--Banach theorem) that ${\rm span}\{\mu_{\Gamma}\} _{\Gamma\in \FA}\cap B_{\M(\KA)}$ is $w^*$-dense in $B_{\M(\KA)}$. Since $\FA$ is dense in $\KA$, it suffices to show that $\{\delta_{\Gamma} \}_{\Gamma\in \FA}\subseteq {\rm span}\{\mu_{\Gamma}\}_{\Gamma\in \FA}$. Indeed, we show by induction that $\{\delta_{\Gamma}\}_{\Gamma\in \FA,|\Gamma|\leq n}\subseteq {\rm span}\{\mu_{\Gamma} \}_{\Gamma\in \FA}$, for every $n<\omega$.
\begin{itemize}
    \item For $n=0$, we just have $\delta_{\emptyset}= \mu_{\emptyset}\in {\rm span}\{\mu_{\Gamma}\}_{\Gamma\in \FA}$.
    \item Inductively, suppose that $n\geq1$ and $\{\delta_{\Gamma}\} _{\Gamma\in \FA,|\Gamma|\leq n-1}\subseteq {\rm span}\{\mu_{\Gamma}\}_{\Gamma\in \FA}$. If $\Gamma \in \FA$ is such that $|\Gamma|=n$, then $$\delta_{\Gamma}=\mu_\Gamma - \sum_{j=0}^{n-1}(-1)^{n-j}\sum_{\Delta\in[\Gamma]^j}\delta_{\Delta}\in {\rm span}\{\mu_{\Gamma}\}_{\Gamma\in \FA},$$
    by the inductive assumption.
\end{itemize}
Therefore, $\{f_{\Gamma};\mu_{\Gamma}\}_{\Gamma\in\FA}$ is a 1-norming M-basis for $\C(\KA)$, as desired.
\end{proof}

\begin{remark} We denote by $\sigma_1(\Gamma)$ the one-point compactification of the discrete set $\Gamma$. It is fairly easy to see that $\sigma_1(\Gamma)^\omega$ is (homeomorphic to) an adequate compact, see, \emph{e.g.}, \cite{P}. Therefore, our previous result generalises, with a similar (but cleaner) proof, \cite[Theorem 3]{Hajek}, where a $1$-norming M-basis is constructed in $\C(\sigma_1(\Gamma)^\omega)$.

Moreover, let us observe that in case $\KA$ is scattered, namely every set in $\A$ is finite, the M-basis constructed in the proof of Theorem \ref{t: adequate implies 1-norming} is even shrinking, as it is not hard to see from the above argument. This is, of course, in complete accordance with Theorem \ref{Th: shrinking M-basis}.
\end{remark}

In conclusion to our note, let us recall the classical result that $\C(2^\omega)$ has no unconditional basis, \cite{Karlin}. Here, $2^\omega$ denotes the Cantor set that, in our notation, is merely $\P(\omega)$, an adequate compact. In particular, the M-basis constructed in Theorem \ref{t: adequate implies 1-norming} is, in general, not unconditional.\medskip

{\bf Acknowledgements.} The authors wish to express their gratitude to Mari\'an Fabian and Gilles Godefroy for their insightful remarks on the problem considered in the article. Such remarks were extremely useful when preparing the final version of the manuscript.


\end{document}